\documentclass[12pt, leqno]{amsart} 

\usepackage[mathcal]{euscript}
\usepackage{amssymb, amsfonts, xy}
\usepackage{graphicx}

\xyoption{all} \CompileMatrices

\begin{document}

\def\colim{\mathop{\rm colim}}
\def\ker{\mathop{\rm ker}\nolimits}
\def\precdot{\mathop{\prec\!\!\!\cdot}\nolimits}
\def\hom{\mathop{\rm Hom}\nolimits}
\def\Aut{\mathop{\rm Aut}}

\def\presup#1{{}^{#1}\hspace{-.11em}}      
\def\presub#1{{}_{#1}\hspace{-.08em}}      

\def\propaorder{P1}   
\def\propbnn{P2}      
\def\propb2e{P3}      
\def\propxxa{\ref{remark-xxa}}      
\def\propcommutes{P4} 
\def\propautker{P5}   
\def\prophomzp{P6}    
\def\propsplits{\ref{remark-splits}}   
\def\propna{P7}       
\def\propbnb{P8}     

\newtheorem{thm}{Theorem}[section]
\newtheorem{theorem}[thm]{Theorem}
\newtheorem{lemma}[thm]{Lemma}
\newtheorem{corollary}[thm]{Corollary}
\newtheorem{proposition}[thm]{Proposition}
\newtheorem{example}[thm]{Example}

\theoremstyle{remark}
\newtheorem{rem}[thm]{Remark}

\makeatletter
\let\c@equation\c@thm
\makeatother
\numberwithin{equation}{section}

\newcommand{\comment}[1]{}
\newcommand{\case}[1]{{\bf \par \noindent Case #1:}}
\newcommand{\mylabel}[1]{\label{#1}}



\title{Large localizations of finite groups}
\author{Adam J. Prze\'zdziecki $^1$}
\address{Institute of Mathematics, Warsaw University, Warsaw, Poland}
\address{Warsaw University of Life Sciences - SGGW, Warsaw, Poland}
\email{adamp@mimuw.edu.pl}

\maketitle

\comment{
\begin{center}
\vspace{-0.7cm}
{\small \it Institute of Mathematics, Warsaw University, Warsaw, Poland \\
Email address: \verb+adamp@mimuw.edu.pl+ } \\
\end{center}
}

\footnotetext[1]{The author was partially supported by grant
  1 P03A 005 26}

\begin{abstract}
  We construct examples of localizations
  in the category of groups which take the Mathieu group $M_{11}$ to
  groups of arbitrarily large cardinality which are ``abelian
  up to finitely many generators''. The paper is part of
  a broader study on the group theoretic properties which are or are not
  preserved by localizations.

  \vspace{4pt}
  {\noindent\leavevmode\hbox {\it MSC:\ }}
  {\bf 20J15}(20D99)
  \vspace{-9pt}
\end{abstract}


\section{Introduction}
\mylabel{section-introduction}

Let $f:H\to G$ be a group homomorphism. We say (cf.
\cite{aschbacher}) that $f$ is {\em closed} if it induces via
composition a bijection of sets
$$\hom(G,G)\to\hom(H,G)$$
This paper is part of a broader study of the question regarding
which group theoretic properties do and which do not pass from $H$
to $G$ via closed homomorphisms. It is more natural to formulate
this in the language of localizations, as explained in the next
paragraph.

The interest in closed homomorphisms is motivated by the following
fact proved in \cite[Lemma 2.1]{casacuberta-survey}. There exists a
closed homomorphism $f:H\to G$ if and only if there exists a
localization $L:Groups\to Groups$ such that $LH=G$. A {\em
localization} (also called a reflection) is a functor
$L:\mathcal{C}\to\mathcal{C}$ equipped with a natural transformation
$\eta:Id\to L$ such that the compositions $L(\eta_X)$ and
$\eta_{L(X)}$ are equal and are isomorphisms, for each object $X$.
Thus the question we study is: which group theoretic properties are
and which are not preserved by localizations. For a survey of this
problem see \cite{casacuberta-survey}.

In Section \ref{section-examples} we prove that for the Mathieu
group $M_{11}$ and any cardinal $\kappa$ there exist closed
inclusions $M_{11}\to G$ such that the cardinality of $G$ is at
least $\kappa$, $G$ is generated by $M_{11}$ and an abelian
subgroup, and the abelianization of $G$ is of the same cardinality
as $G$, in particular $G$ is very far from being simple. This
example is obtained as a corollary of Theorem \ref{theorem-main}
which states conditions on a finite group $S$ that imply the
existence of such closed embeddings $S\to G$. We believe it should
be possible to construct a solvable group $S$ which meets these
conditions.

First examples of closed embeddings of a finite group into an
infinite one were described by Libman \cite{libman}. Closed
embeddings of finite simple groups into arbitrarily large simple
groups were constructed by G\"{o}bel, Rodr\'iguez and Shelah in
\cite{grs} and \cite{gs}. Nonsimple localizations of finite simple
groups were described by Rodr\'iguez, Scherer and Viruel in
\cite{rsv-simple}. Nonperfect localizations of infinite perfect
groups were obtained independently by Badzioch and Feshbach in
\cite{badzioch} and by Rodr\'iguez, Scherer and Viruel in
\cite{rsv-perfect}.

Closer to the ``abelian end'', Dwyer and Farjoun have asked
whether a closed homomorphism $f:H\to G$ with $H$ finite nilpotent
must be an epimorphism. They proved that this is the case when $H$
is of nilpotency class at most two (see \cite[Theorem
3.3]{libman-nilpotent} and \cite[Theorem
2.3]{casacuberta-survey}). Aschbacher \cite{aschbacher} extended
this result to the case where $H$ is of class at most three, under
the additional assumption that $G$ is finite.

The author wishes to express his thanks to the referee for many
suggestions, especially for removing a computer aided verification
from the proof of Example \ref{example-mathieu} and two Properties
from the list in Section \ref{section-construction}

\section{Bass-Serre theory}

In this section we collect elements of the Bass-Serre theory on
group amalgams acting on trees. Most of the results below can be
found in \cite{serre}, some in \cite{magnus} and \cite{brown},
others are simple corollaries. We follow the terminology of
\cite{serre}.

The neutral element of a group is always denoted by $e$. If $k$ is
an element and $x$ is an element or a subset of a group we often
write $\presup{k}x$ instead of $kxk^{-1}$.

If $H$ is a common subgroup of $G_1$ and $G_2$ then the {\em
amalgam} $G_1*_HG_2$ is the colimit of the diagram $G_1\supseteq
H\subseteq G_2$, that is, the unique, up to isomorphism, group $G$
which contains $G_1$ and $G_2$ and such that $G_1\cap G_2=H$ and
for any pair of homomorphisms $f_1:G_1\to K$ and $f_2:G_2\to K$ if
$f_1|_H=f_2|_H$ then there exists a unique homomorphism $f:G\to K$
extending both $f_1$ and $f_2$. It is customary to call $G_1$ and
$G_2$ the {\em factors} of the amalgam $G_1*_HG_2$.

\rem \mylabel{remark-amalgam-generation}
  The amalgam $G$ is generated by $G_1\cup G_2$.

\begin{lemma}\cite[\S 4.1 Theorem 7]{serre}
\mylabel{lemma-tree}
  Let $G=G_1*_HG_2$ be an amalgam of groups.
  There exists a unique, up to isomorphism, tree $T_G$ on which
  $G$ acts with fundamental domain a segment:
\begin{center}
\def\piccent{20}
\begin{picture}(96,28)
  \put(23,\piccent){\circle{4}}
  \put(73,\piccent){\circle{4}}
  \put(25,\piccent){\line(1,0){46}}
  \put(23,\piccent){\put(0,-16){\makebox(0,0)[b]{$G_1$}}}
  \put(23,\piccent){\put(25,-12){\makebox(0,0)[b]{$H$}}}
  \put(23,\piccent){\put(50,-16){\makebox(0,0)[b]{$G_2$}}}
\end{picture}
\end{center}
The labels denote the stabilizers of the edge and its vertices. No
element of $G$ may swap the ends of an edge of $T_G$.
\end{lemma}
In this paper, as in \cite{serre}, the group $G$ acts on $T_G$ from
the left. Below we list some immediate consequences of
\ref{lemma-tree}.

\rem \mylabel{remark-tree-consequences}
\begin{itemize}

  \item[(i)] The stabilizers of the vertices of $T_G$ are precisely the
    conjugates of either $G_1$ or $G_2$.
  \item[(ii)] The vertex stabilizer acts transitively on the
    neighbors of that vertex (since $G$ acts transitively on the
    edges of $T_G$).
  \item[(iii)] The assignment of a stabilizer to a vertex of $T_G$
    is a one-to-one correspondence, which
    allows us to make no distinction between vertices and their
    stabilizers. This means that we identify the vertex set of
    $T_G$ with $\{\presup{g}G_i \mid i=1,2,g\in G\}$ and the edge
    set of $T_G$ with
    $\{\{\presup{g}G_1,\presup{g}G_2\}\mid g\in G\}$.
    Then $G$ acts (on the left) on $T_G$ via (left) conjugation.
    In particular, an $x$ in $G$ fixes a vertex $V$ if and
    only if $x\in V$.
  \item[(iv)] If the graph
\begin{center}
\def\piccent{20}
\begin{picture}(146,28)
  \put(23,\piccent){\circle{4}}
  \put(73,\piccent){\circle{4}}
  \put(123,\piccent){\circle{4}}
  \put(25,\piccent){\line(1,0){46}}
  \put(75,\piccent){\line(1,0){46}}
  \put(23,\piccent){\put(0,-16){\makebox(0,0)[b]{$G_1$}}}
  \put(23,\piccent){\put(25,-12){\makebox(0,0)[b]{$H$}}}
  \put(23,\piccent){\put(50,-16){\makebox(0,0)[b]{$G_2$}}}
  \put(23,\piccent){\put(75,-12){\makebox(0,0)[b]{$H'$}}}
  \put(23,\piccent){\put(100,-16){\makebox(0,0)[b]{$G_2'$}}}
\end{picture}
\end{center}
is a fragment of $T_G$ then, as in (ii), there exists an element
$g\in G_2$ such that $G_1{\,\!\!\!}'=\presup{g}G_1$. Moreover for
the same $g$ we have $H'=\presup{g}H$ and $G_1\cap G_1'=H\cap H'$.
\end{itemize}

The next lemma defines the notion of a {\em reduced decomposition}
of an element $g\in G_1*_HG_2$. We choose sets $R(G_1)$ and
$R(G_2)$ of right coset representatives of $H\backslash G_1$ and
$H\backslash G_2$ respectively such that $e$ belongs to $R(G_1)$
and $R(G_2)$.

\begin{lemma}
\mylabel{lemma-reduced-decomposition}
  For every $g\in G_1*_HG_2$ there exists a unique decomposition
  (called a reduced decomposition)
  $$ g=ar_1r_2\ldots r_n $$
  such that $a\in H$ and $r_i\in R(G_1)\cup R(G_2)\smallsetminus\{e\}$ for
  $i=1,2,\ldots,n$ and for any $i=1,2,\ldots,n-1$ one of $r_i$ and
  $r_{i+1}$ belongs to $R(G_1)$ and the other to $R(G_2)$.
\end{lemma}

\begin{proof}
  See \cite[\S 1.2, Theorem 1]{serre}.
\end{proof}

The integer $n$ above is called the {\em length} of $g$ and is
denoted $l(g)$. It is easy to see that $g$ decomposes as in Lemma
\ref{lemma-reduced-decomposition} if and only if $g=g_0g_1\dots
g_n$, where $g_0\in H$ and $g_i$, for $i=1,2,\dots n$, alternately
belongs to $G_1\smallsetminus H$ and $G_2\smallsetminus H$.
Therefore $l(g)$ does not depend on the choice of the right coset
representatives and we have $l(g^{-1})=l(g)$. An element
$g=ar_1r_2\ldots r_n$ is called {\em cyclically reduced} if
$l(g)\geq 2$ and one of $r_1$, $r_n$ belongs to $G_1\smallsetminus
H$ and the other to $G_2\smallsetminus H$. We note that $g\in G$
is cyclically reduced if and only if $l(g)\geq 2$ is even.

\begin{lemma}
\mylabel{lemma-c-reduced}
  Every element of $G_1*_HG_2$ is conjugate to a cyclically reduced one or
  to an element of $G_1\cup G_2$. Every cyclically reduced element is
  of infinite order.
\end{lemma}

\begin{proof}
  See \cite[\S1.3 Proposition 2]{serre}.
\end{proof}

\begin{lemma}
\mylabel{lemma-cyclically-reduced-conjugate}
  If $g$ is a cyclically reduced element of $G_1*_HG_2$ which is
  conjugate to an element of the form $r_1r_2\ldots r_k$,
  where $k\geq 2$
  and $r_i$, $r_{i+1}$ as well as $r_1$, $r_k$ are in distinct
  factors then $g$ can be obtained by cyclically permuting
  $r_1,r_2,\ldots r_k$ and then conjugating by an element of $H$.
\end{lemma}

\begin{proof}
  See \cite[Theorem 4.6(iii)]{magnus}.
\end{proof}

In particular, in view of the description of the length that
follows Lemma \ref{lemma-reduced-decomposition}, we have:

\begin{lemma}
\mylabel{lemma-length-cyclically-reduced-conjugate}
  If $g$ and $h$ are conjugate cyclically reduced elements of
  $G_1*_HG_2$ then $l(g)=l(h)$.
\end{lemma}

If $P$, $Q$ are vertices of a tree then $l(P,Q)$ is the length
(i.e. the number of edges) of the shortest path from $P$ to $Q$
and is called the {\em distance} from $P$ to $Q$. The shortest
path is called a {\em geodesic}.

\begin{lemma}
\mylabel{lemma-g-geodesic}
  If $g=ar_1r_2\ldots r_n$ is a reduced decomposition of
  $g$ in $G_1*_HG_2$ with $r_1$ in $G_1$
  then the geodesic from $G_2$ to
  $\presup{g^{-1}}G_2$ contains the vertices: \\
  $G_2$, $\presup{r_n^{-1}}G_1$, $\presup{r_n^{-1}r_{n-1}^{-1}}G_2$,
  \ldots , $\presup{r_n^{-1}r_{n-1}^{-1}\ldots r_1^{-1}}G_2$ if $n$
  is even, and \\
  $G_2$, $G_1$, $\presup{r_n^{-1}}G_2$, $\presup{r_n^{-1}r_{n-1}^{-1}}G_1$,
  \ldots , $\presup{r_n^{-1}r_{n-1}^{-1}\ldots r_1^{-1}}G_2$ if $n$
  odd.
\end{lemma}

\begin{proof}
  It is enough to observe that each $r_i$ rotates the edge $H$ about
  one of its end points and argue by induction: for $n=0$ the
  claim is true. If $n$ is odd then the case of $n-1$ implies that
  $\presup{r_n^{-1}}G_2$, $\presup{r_n^{-1}r_{n-1}^{-1}}G_1$,
  \ldots , $\presup{r_n^{-1}r_{n-1}^{-1}\ldots r_1^{-1}}G_2$
  is a geodesic, and
  the claim follows since any path without backtracking in a tree
  is a geodesic. If $n$ is even then, analogously,
  $\presup{r_n^{-1}}G_1$, $\presup{r_n^{-1}r_{n-1}^{-1}}G_2$,
  \ldots , $\presup{r_n^{-1}r_{n-1}^{-1}\ldots r_1^{-1}}G_2$ is a
  geodesic and the lemma follows.
\end{proof}

\begin{lemma}
\mylabel{lemma-normal-subgroup}
  Suppose that $H_0\subseteq H$ is such a subgroup that for
  $x$ in $G_1\cup G_2$
  an inclusion $\presup{x}H_0\subseteq H$ implies
  $\presup{x}H_0=H_0$.
  Then the normalizer of $H_0$ in $G_1*_HG_2$ may be presented as
  $$
   N_{G_1*_HG_2}(H_0)=N_{G_1}(H_0)*_HN_{G_2}(H_0)
  $$
\end{lemma}

\begin{proof}
  The right hand side is well defined since $H_0$ is normal in
  $H$.
  Only the inclusion $N_{G_1*_HG_2}(H_0)\subseteq N_{G_1}(H_0)*_HN_{G_2}(H_0)$
  is not obvious. Let $g=ar_1r_2\ldots r_n$ be the reduced
  decomposition of an element $g$ in $N_{G_1*_HG_2}(H_0)$. We have
  $\presup{g}H_0=H_0\subseteq G_1\cap G_2$ hence
  $$
    H_0\subseteq\presup{g^{-1}}G_1\cap\presup{g^{-1}}G_2\cap G_1\cap G_2
  $$
  and
  therefore the elements of $H_0$ fix all the vertices of the geodesics
  which connect $G_2$ or $G_1$ to $\presup{g^{-1}}G_2$ or
  $\presup{g^{-1}}G_1$. By Lemma \ref{lemma-g-geodesic} we have
  $H_0\subseteq\presup{r_n^{-1}r_{n-1}^{-1}\ldots r_i^{-1}}G_2$ and
  $H_0\subseteq\presup{r_n^{-1}r_{n-1}^{-1}\ldots r_i^{-1}}G_1$ for
  $i=1,2,\ldots,n$, hence
  $\presup{r_ir_{i+1}\ldots r_n}H_0\subseteq G_1\cap G_2=H$ and
  therefore, by a downward induction, $r_i\in N_{G_1}(H_0)\cup N_{G_2}(H_0)$
  for $i=n,n-1,\ldots,1$.
\end{proof}

\begin{lemma}
\mylabel{lemma-conjugate-elements}
  Suppose that $G=E*_HK$ is such that if $x,y\in H$ are conjugate in
  $E$ then they are conjugate in $K$. Then if $x,y\in K$ are
  conjugate in $G$ then they are conjugate in $K$.
\end{lemma}

\begin{proof}
  See \cite[Theorem 4.6 (i) and (ii)]{magnus}.
  \comment{
  Suppose, to the contrary, that some $x,y\in K$ are conjugate in
  $G$ but not in $K$. Let $g\in G$ be such that $\presup{g}x=y$ and
  $n=l(g)$ is minimal. In the reduced decomposition $g=ar_1r_2\ldots
  r_n$ we have $r_1\notin K$ and $r_n\notin K$ hence $n$ is odd. We have
  $x\in K$ and $\presup{g}x\in K$, hence $x\in\presup{g^{-1}}K$,
  and therefore $x$ fixes the vertices of the geodesic from $K$ to
  $\presup{g^{-1}}K$. Lemma \ref{lemma-g-geodesic}
  implies that $x\in\presup{r_n^{-1}}K$, that is, $\presup{r_n}x\in
  K$. There exists $r_n'\in K$ such that
  $\presup{r_n'}x=\presup{r_n}x$. If $r_n$ is replaced by $r_n'$ we
  obtain $g'$ such that $\presup{g'}x=y$ and $l(g')<l(g)$, which is a
  contradiction.
  }
\end{proof}

\begin{lemma}
\mylabel{lemma-tree-fixed-points}
  If $g\in G=G_1*_HG_2$ then the action of $g$ on $T_G$ satisfies one
  of the following:
  \begin{itemize}
    \item[a)] $g$ has no fixed points.
    \item[b)] $g$ fixes a unique vertex of $T_G$.
    \item[c)] $g$ fixes a conjugate of $G_1$ and a conjugate of
    $G_2$.
  \end{itemize}
\end{lemma}
\begin{proof}
  If a) and b) do not hold then $g$ fixes two points $P$ and $Q$.
  By the uniqueness it has to fix the geodesic from $P$ to $Q$.
  The proof is complete since the vertices of every path are
  conjugates of $G_1$ and $G_2$ alternately.
\end{proof}

\begin{lemma}
\mylabel{lemma-finite-conjugate}
  Let $F\subseteq G_1*_HG_2$ be a finite subgroup. Then $F$ is
  conjugate to a subgroup of $G_1$ or $G_2$.
\end{lemma}

\begin{proof}
  See \cite[Chapter II, Corollary A3]{brown}.
\end{proof}

A doubly infinite chain
\begin{center}
\def\piccent{12}
\begin{picture}(246,29)
  \put(23,\piccent){\circle{4}}
  \put(73,\piccent){\circle{4}}
  \put(123,\piccent){\circle{4}}
  \put(173,\piccent){\circle{4}}
  \put(223,\piccent){\circle{4}}
  \put(6,\piccent){\line(1,0){15}}
  \put(25,\piccent){\line(1,0){46}}
  \put(75,\piccent){\line(1,0){46}}
  \put(125,\piccent){\line(1,0){46}}
  \put(175,\piccent){\line(1,0){46}}
  \put(225,\piccent){\line(1,0){15}}
  \put(260,\piccent){\makebox(0,0){$\ldots$}}
  \put(-14,\piccent){\makebox(0,0){$\ldots$}}
\end{picture}
\end{center}
is called a {\em straight path}.

\begin{lemma}
\mylabel{lemma-serre-24}
  Let $s$ be an automorphism acting on a tree $X$ with no fixed
  points. Let
  $$
  m = \inf_{P\in \mathop{\rm vert} X} l(P,sP) \mbox{\ \ and\ \ }
  T = \{P\in \mathop{\rm vert} X \mid l(P,sP)=m\}
  $$
  Then:
  \begin{itemize}
    \item[i)] $T$ is the vertex set of a straight path of $X$.
    \item[ii)] $s$ induces a translation of $T$ of amplitude $m$.
    \item[iii)] Every subtree of $X$ stable under $s$ and $s^{-1}$
    contains $T$.
    \item[iv)] If a vertex $Q$ of $X$ is at a distance $d$ from
    $T$ then $l(Q,sQ)=m+2d$.
  \end{itemize}
\begin{center}
\def\piccent{14}
\begin{picture}(246,80)
  \put(23,\piccent){\put(0,50){\circle{4}}}
  \put(223,\piccent){\put(0,50){\circle{4}}}
  \put(23,\piccent){\circle{4}}
  \put(73,\piccent){\circle{4}}
  \put(123,\piccent){\circle{4}}
  \put(173,\piccent){\circle{4}}
  \put(223,\piccent){\circle{4}}
  \put(23,\piccent){\put(0,2){\line(0,1){46}}}
  \put(223,\piccent){\put(0,2){\line(0,1){46}}}
  \put(25,\piccent){\line(1,0){46}}
  \put(75,\piccent){\line(1,0){46}}
  \put(125,\piccent){\line(1,0){46}}
  \put(175,\piccent){\line(1,0){46}}
  \put(225,\piccent){\line(1,0){15}}
  \put(6,\piccent){\line(1,0){15}}
  \put(123,\piccent){\put(-11,7){\vector(1,0){22}}}
  \put(23,\piccent){\put(0,-16){\makebox(0,0)[b]{$P'$}}}
  \put(223,\piccent){\put(0,-16){\makebox(0,0)[b]{$sP'$}}}
  \put(123,\piccent){\put(0,10){\makebox(0,0)[b]{$s$}}}
  \put(23,\piccent){\put(-5,50){\makebox(0,0)[r]{$Q$}}}
  \put(223,\piccent){\put(5,50){\makebox(0,0)[l]{$sQ$}}}
  \put(260,\piccent){\makebox(0,0){$\ldots$}}
  \put(-14,\piccent){\makebox(0,0){$\ldots$}}
\end{picture}
\flushright{($m=4, d=1$)}
\end{center}

\end{lemma}

\begin{proof}
  See \cite[\S 6.4, Proposition 24]{serre}.
\end{proof}

\begin{lemma}
\mylabel{lemma-length-cyclically-reduced}
  Let $g$ be a cyclically reduced element of $G=G_1*_HG_2$.
  Let $X=T_G$ and $m$ and $T$ be as in Lemma \ref{lemma-serre-24},
  where $s$ is the action of $g$ on $T_G$.
  Then $l(g)=m$ and both $G_1$ and $G_2$ belong to $T$.
\end{lemma}

\begin{proof}
  Let $g=ar_1r_2\ldots r_n$ be the reduced decomposition of $g$.
  Possibly swapping $G_1$ and $G_2$,
  we may assume that $r_1\in G_1$.
  Lemma \ref{lemma-g-geodesic} implies that
  $l(g)=l(G_2,\presup{g^{-1}}G_2)$. Let $[G_2,\presup{g^{-1}}G_2]$ be the geodesic
  from $G_2$ to $\presup{g^{-1}}G_2$. Let
  $T_0=\bigcup_{k\in\mathbb{Z}}\presup{g^k}[G_2,\presup{g^{-1}}G_2]$.
  The union of consecutive geodesics
  $[\presup{g^{k+2}}G_2,\presup{g^{k+1}}G_2]\cup[\presup{g^{k+1}}G_2,\presup{g^k}G_2]$
  is a path of length $2l(g)$ that connects $\presup{g^k}G_2$ and
  $\presup{g^k}(\presup{g^2}G_2)$. Since $g$ is cyclically reduced
  we see that $l(g^2)=2l(g)$. Lemma \ref{lemma-g-geodesic} implies
  that the geodesic connecting $\presup{g^k}G_2$ and
  $\presup{g^k}(\presup{g^2}G_2)$ has length $l(g^2)$ hence the
  path above is a geodesic and consequently, since $T_G$ is a
  tree, the union $T_0$ is a straight path on which $g$ acts as
  translation by
  $l(g)$. Since $T_0$ is invariant under $g$ and
  $g^{-1}$, Lemma \ref{lemma-serre-24}(iii) implies that
  $T\subseteq T_0$, so since they are both straight paths, we have
  $T=T_0$, hence $G_2\in T$ and $l(g)=m$.
  We prove that $G_1\in T$ by repeating the argument above with
  $g$ replaced by $g^{-1}$ and $G_1$ swapped for $G_2$.
\end{proof}

\section{Large $E$-rings}
\mylabel{section-e-rings}

This section describes the ``filling'', which makes targets of our
closed embeddings arbitrarily large. The reader interested in
``infinite'' rather than ``arbitrarily large'' may take
$E=\mathbb{Z}_{(q)}=\{{m\over n}\in\mathbb{Q}\mid q\nmid n\}$.

The notion of an $E$-ring was introduced by Schultz
\cite{schultz}. Let $\hom_R(E,E)$ denote the ring of endomorphisms
of $E$ as a right $R$-module. A ring $E$ with identity $1$ is said
to be an {\em $E$-ring} if the ring restriction homomorphism
$$
  E\cong\hom_E(E,E)\to\hom_\mathbb{Z}(E,E)
$$
is an isomorphism. This forces $E$ to be commutative. Unless
explicitly stated, we work only with the additive group of $E$ and
denote it with the same symbol $E$.

With the terminology outlined in Section
\ref{section-introduction} above we may characterize the additive
groups of $E$-rings as those which admit a nontrivial closed
homomorphism $\mathbb{Z}\to E$, or equivalently, as possible
values of group localizations of the integers.

We make use of a particular class of examples of $E$-rings,
constructed by Dugas, Mader and Vinsonhaler \cite{erings}. The
following theorem is extracted from \cite{erings}.

\begin{thm}
\mylabel{thm-ering}
For any prime number $q$ and an infinite cardinal number $\kappa$,
not strictly between $\aleph_0$ and the continuum, there exists an
abelian group $E$ of cardinality $\kappa$ with the following
properties:
\begin{itemize}
  \item[(1)]
    $E$ is the additive group of an {\em $E$-ring}.
  \item[(2)]
    $E$ is torsion free.
  \item[(3)]
    No nonzero element of $E$ is divisible by all powers of $q$.
  \item[(4)]
    $E$ is $p$-divisible for every prime $p\neq q$.
  \item[(5)]
    All nontrivial endomorphisms $f:E\to E$ are injective.
\end{itemize}
\end{thm}
Item (5) of Theorem \ref{thm-ering} is not explicitly stated in
\cite{erings} but $E$ is constructed as a ring with no zero
divisors (i.e. elements $a\neq 0$, $b\neq 0$ such that $ab=0$).
Since each endomorphism of an $E$-ring is a multiplication by an
element of $E$, item (5) follows.

\section{Construction of closed embeddings}
\mylabel{section-construction}
In this section we construct arbitrarily large groups $L$ and
closed embeddings $S\to L$ with finite source $S$.

Let $S$ be a finite group with no outer automorphisms. Let $a$ and
$b$ be two elements of $S$. Let $A=\langle a \rangle$ be the
cyclic subgroup of $S$ generated by $a$ and let $N=N_S(A)$ be its
normalizer in $S$. We fix a prime number $p>2$ and assume the
following.

\noindent {\bf Properties:}
\begin{itemize}
  \item[\propaorder{.}]
    The order of $a$ is $p$.
  \item[\propbnn{.}]
    The element $b$ is not in $N$.
  \item[\propb2e{.}] $b^2=e$.
  \item[\propcommutes{.}]
    If $k\in S$ commutes with $a$ and $b$ then $k=e$.
  \item[\propautker{.}]
    Any homomorphism $f:S\to S$ is either an automorphism or
        contains $a$ and $b$ in its kernel.
  \item[\prophomzp{.}]
    $S$ contains no element of order $p^2$.
  \item[\propna{.}] $p$ does not divide the order of $N/A$.
  \item[\propbnb{.}]
    The intersection $N\cap\presup{b}N$ is trivial.
\end{itemize}

\begin{rem}
\mylabel{remark-xxa}
  Property {\propbnb} implies that for any $k\in N$
  if $kbk^{-1}b^{-1}\in A$ then $k=e$.
\end{rem}

\begin{rem}
\mylabel{remark-splits}
  Property {\propna} and the Schur-Zassenhauss Theorem imply that
  the exact sequence $\{e\}\to A\to N\to N/A\to\{e\}$
  splits.
\end{rem}

\begin{rem}
\mylabel{remark-E-S}
  For any infinite cardinal number $\kappa$, not strictly between
  $\aleph_0$ and the continuum, there exists an $E$-ring $E$,
  as in Theorem \ref{thm-ering}, which is $p$-divisible for every
  prime $p$ dividing the order of $S$.
  To obtain such an $E$ it is enough to choose the prime $q$
  so that it does not divide the order of $S$
  and use Theorem \ref{thm-ering}(4).
\end{rem}

Let $C=\langle c \rangle$ be a cyclic group of order $p^2$. We
identify $A$ with the subgroup of $C$ generated by $c^p$. The
restriction homomorphism $\Aut(C)\to\Aut(A)$ splits uniquely,
hence the split exact sequence in Remark {\propsplits} extends to
a split sequence:
$$
  \{e\}\to C\to M\to M/C\to\{e\}
$$
where $M/C\cong N/A$. We view $N$ as a subgroup of $M$. We note
that since $p>2$ the unique split of $\Aut(C)\to\Aut(A)$ takes
$a\mapsto a^{-1}$ to $c\mapsto c^{-1}$.

Let
$$K=M*_N S$$
be an amalgam of groups. By \ref{lemma-tree} there exists a
unique, up to isomorphism, tree $T_K$ on which $K$ acts with
fundamental domain
\begin{center}
\def\piccent{20}
\begin{picture}(96,30)
  \put(23,\piccent){\circle{4}}
  \put(73,\piccent){\circle{4}}
  \put(25,\piccent){\line(1,0){46}}
  \put(23,\piccent){\put(0,-16){\makebox(0,0)[b]{$M$}}}
  \put(23,\piccent){\put(25,-12){\makebox(0,0)[b]{$N$}}}
  \put(23,\piccent){\put(50,-16){\makebox(0,0)[b]{$S$}}}
\end{picture}
\end{center}
where the labels denote the stabilizers of the edge and its
vertices.

In a similar way we define
$$L=E*_ZK$$
where the group $E$ is chosen as in Remark \ref{remark-E-S} and
$Z\subseteq E$ is the subgroup generated by the ring identity
$1\in E$ and is identified with $\langle cb \rangle \subseteq K$.
Let $T_L$ be the tree which corresponds to the amalgam $L$. We
denote by $\eta:S\to L$ the inclusion of $S$ into $L$. The
remainder of this paper is devoted to the proof that $\eta$ is
closed and that $M_{11}$ satisfies Properties P1--P8.

\section{Properties of the construction}
\mylabel{section-properties}
In this section we describe some properties of the inclusion
$\eta:S\to L$, introduced in Section \ref{section-construction},
which are used in Section \ref{section-proof}.

\begin{rem}
  By the construction,
  $M$ is generated by $N\cup\{c\}$. Remark
  \ref{remark-amalgam-generation} implies that
  the group $K$ is generated by $S\cup\{c\}$.
  The group $L$ is generated by $K$ and $E$ and therefore
  by $S$ and $E$.
\end{rem}

\begin{lemma}
\mylabel{lemma-normalizer} Let $g=(cb)^v$ for some $v>0$.
  The normalizer of $\langle g \rangle$ in $K$ is
  $\langle cb \rangle$.
\end{lemma}

\begin{proof}
  Since $c\in M\smallsetminus N$ and $b\in S\smallsetminus N$
  we see that $g$ is cyclically reduced of length $2v$ in
  $K=M*_NS$. Lemmas \ref{lemma-serre-24} and
  \ref{lemma-length-cyclically-reduced} imply that there exists a
  unique straight path $T\subseteq T_K$, stable under the action of
  $\langle g\rangle$; moreover $M$ and $S$ belong to $T$ and the
  action of $g$ restricted to $T$ is a translation of amplitude
  $l(g)=2v$.

  If $k\in N_K(\langle g\rangle)$ then $kgk^{-1}=g^\varepsilon$
  for some
  $\varepsilon\in\{-1,1\}$, hence $kgk^{-1}$ stabilizes the unique
  path $T$
  above and therefore $\presup{gk^{-1}}T=\presup{k^{-1}}T$,
  hence again by the uniqueness of $T$ as in Lemma
  \ref{lemma-serre-24}(iii) we have $\presup{k^{-1}}T=T$.
  We conclude that $N_K(\langle g\rangle)$ stabilizes $T$.

  Let $\phi:N_K(\langle g\rangle)\to\Aut(T)$ be the
  homomorphism obtained by restricting the automorphisms of $T_K$ to
  automorphisms of $T\subseteq T_K$. In this paragraph we prove that
  $\phi$ is one-to-one.
  We draw a part of $T$:
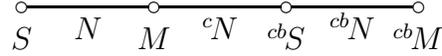
  \begin{center}
  \def\piccent{20}
  \begin{picture}(196,30)
    \put(23,\piccent){\circle{4}}
    \put(73,\piccent){\circle{4}}
    \put(123,\piccent){\circle{4}}
    \put(173,\piccent){\circle{4}}
    \put(25,\piccent){\line(1,0){46}}
    \put(75,\piccent){\line(1,0){46}}
    \put(125,\piccent){\line(1,0){46}}
    \put(23,\piccent){\put(0,-16){\makebox(0,0)[b]{$S$}}}
    \put(23,\piccent){\put(25,-12){\makebox(0,0)[b]{$N$}}}
    \put(73,\piccent){\put(0,-16){\makebox(0,0)[b]{$M$}}}
    \put(73,\piccent){\put(25,-12){\makebox(0,0)[b]{$\presup{c}N$}}}
    \put(123,\piccent){\put(0,-16){\makebox(0,0)[b]{$\presup{cb}S$}}}
    \put(123,\piccent){\put(25,-12){\makebox(0,0)[b]{$\presup{cb}N$}}}
    \put(173,\piccent){\put(0,-16){\makebox(0,0)[b]{$\presup{cb}M$}}}
  \end{picture}
  \end{center}
  Since $\ker\phi$ acts trivially on $T$ we have
  $\ker\phi\subseteq \presup{c}N\cap\presup{cb}N$. Property
  {\propbnb} implies that $N\cap\presup{b}N=\{e\}$, hence
  $\ker\phi\subseteq \presup{c}(N\cap\presup{b}N)=\{e\}$.

  Since $\langle cb\rangle\subseteq N_K(\langle g\rangle)$,
  the action of $N_K(\langle g\rangle)$ on the vertices of
  $T$ has two orbits: the $S$ conjugates and the $M$ conjugates,
  and therefore it is enough to prove that no element of
  $N_K(\langle g\rangle)$ acts on $T$ as a reflection at $S$. Suppose to the
  contrary that $x\in N_K(\langle g\rangle)$ is such an element.
  Then $x\in S$ since $x$ fixes $S$, and $x^2=e$ since $\ker\phi$ is
  trivial. Also $xcbx^{-1}=(cb)^{-1}$, hence $cbxcbx=e$. Since
  $cbx$ is torsion Lemma \ref{lemma-finite-conjugate} implies that
  it belongs to a conjugate of $S$ or $M$. Since
  $bx\in S$ we have $l(cbx)\leq 2$, hence $cbx\in S\cup M$. Since
  $c\in M\smallsetminus N$ we have $cbx\in M$, hence $bx\in S\cap M=N$.
  Remark {\propsplits} implies that
  $N=A\rtimes(N/A)$, hence $bx=a_0n$ for some $a_0\in A$ and
  $n\in N/A$. Since $e=(cbx)^2=(ca_0n)^2$ and $M=C\rtimes(N/A)$ we
  have $n^2=e$; hence, as $N=A\rtimes(N/A)$, also
  $(bx)^2=(a_0n)^2\in A$.

  If $(bx)^2=e$ then, since $x^2=e$ and $b^2=e$ (by Property
  {\propb2e}), we see that $b$ commutes with $bx\in N$.
  Since $c^2\neq e$ and $(cbx)^2=e$ we see that $bx\neq e$.
  This contradicts Property \propbnb.

  If $(bx)^2\neq e$ then since the order of $A$ is a prime $p$ we
  see that $(bx)^2$ generates $A$. However $b$ inverts $(bx)^2$
  and hence normalizes $A$, contradicting Property {\propbnn}.
\end{proof}

\begin{lemma}
\mylabel{lemma-trivial-intersection}
  For $k\in L$, if $E\neq \presup{k}E$ then
  $E\cap\presup{k}E=\{e\}$.
\end{lemma}
\begin{proof}
  The path that connects $E$ and $\presup{k}E$ must contain a segment
  conjugate to
  \begin{center}
  \def\piccent{20}
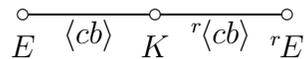
  \begin{picture}(146,28)
    \put(23,\piccent){\circle{4}}
    \put(73,\piccent){\circle{4}}
    \put(123,\piccent){\circle{4}}
    \put(25,\piccent){\line(1,0){46}}
    \put(75,\piccent){\line(1,0){46}}
    \put(23,\piccent){\put(0,-16){\makebox(0,0)[b]{$E$}}}
    \put(23,\piccent){\put(25,-14){\makebox(0,0)[b]{$\langle cb\rangle$}}}
    \put(73,\piccent){\put(0,-16){\makebox(0,0)[b]{$K$}}}
    \put(73,\piccent){\put(25,-14){\makebox(0,0)[b]{$\presup{r}\langle cb\rangle$}}}
    \put(123,\piccent){\put(0,-16){\makebox(0,0)[b]{$\presup{r}E$}}}
  \end{picture}
  \end{center}
  for some $r\in K\smallsetminus\langle cb\rangle$.
  Up to conjugation, we have
  $E\cap \presup{k}E \subseteq
  \langle cb \rangle \cap \presup{r}\langle cb \rangle$.
  If $\langle cb\rangle \cap \presup{r}\langle cb\rangle$ is
  nontrivial then we have $(cb)^s=r(cb)^tr^{-1}$ for some nonzero
  integers $s$ and $t$. Lemma
  \ref{lemma-length-cyclically-reduced-conjugate} applied to $K=M*_NS$
  implies that $|s|=|t|$, hence $r\in N_K(\langle(cb)^s\rangle)$.
  Lemma \ref{lemma-normalizer} implies that
  $N_K(\langle(cb)^s\rangle)=\langle cb\rangle$, which
  contradicts $r\notin\langle cb\rangle$.
\end{proof}

\begin{lemma}
\mylabel{lemma-A-normalizer}
  If $g\in L$ normalizes $A$ then $g\in K$.
\end{lemma}
\begin{proof}
  Since $a\in K$ its action on $T_L$ fixes $K$. Since $a$ is
  torsion and $E$ is torsion-free Lemma
  \ref{lemma-tree-fixed-points} implies that $K$ is the unique
  fixed point of this action. Since $g$ normalizes $A$ we see that
  $\presup{gag^{-1}}K=K$, hence $\presup{ag^{-1}}K=\presup{g^{-1}}K$
  and therefore $\presup{g^{-1}}K=K$ by the uniqueness of the
  fixed point of $a$. The identity $\presup{g^{-1}}K=K$ implies
  $g\in K$.
\end{proof}

\begin{lemma}
\mylabel{lemma-E}
  The image of a homomorphism $f:E\to L$ is conjugate in
  $L$ to a subgroup of $E$.
\end{lemma}

\begin{proof}
  We need to prove that the action of $E$ on $T_L$
  induced by $f$ has a fixed vertex that corresponds to a conjugate
  of $E$. Suppose that $f(E)$ is nontrivial.
  Let $x\in E\smallsetminus \ker f$. If $x$ acts on $T_L$
  without fixed points then by Lemma \ref{lemma-serre-24}
  there exists a straight path $T$ in $T_L$,
  stable under the action of $x$, on which it induces a translation
  by $n>0$ vertices.
  Theorem \ref{thm-ering}(4) implies that $E$ is divisible by many
  primes, hence
  there exist $y\in E$ and an integer $m>n$ such that $y^m=x$.
  The action of $y$ on $T_L$ also has no fixed points and by Lemma
  \ref{lemma-serre-24} again, there exists a straight path $T_y$
  stable under the action of $y$, which $y$ translates by $n_y$
  vertices.
  Since $T_y$ is stable under $x$ and $x^{-1}$, Lemma
  \ref{lemma-serre-24}(iii) implies that $T\subseteq T_y$, so
  since they are both straight paths, $T=T_y$. Now $n_y$ has to be a
  fractional quantity ${n\over m}$, so we obtain a contradiction
  and therefore the set $T^x$ of points fixed by $x$ is nonempty.

  Lemma \ref{lemma-trivial-intersection} implies that $T^x$ may contain
  at most one vertex of the form $\presup{k}E$ for some $k$.
  If $T^x$ does contain a $\presup{k}E$
  then since $E$ is abelian we deduce that $T^x$ is stable
  under the action of $E$ on $T_L$, hence
  $\presup{k}E$ is fixed by this action, that is,
  $f(E)\subseteq\presup{k}E$ as required.
  If $T^x$ does not contain a
  conjugate of $E$ then it consists of precisely one conjugate of
  $K$ and, as above, we have $f(E)\subseteq\presup{k}K$.

  Let $m$ be the order of $S$. Since $E$ was chosen in Remark
  \ref{remark-E-S} to be $m$-divisible and no nontrivial element of $K$
  is divisible by all powers of $m$,
  we obtain a contradiction
  with the assumption that $x\notin\ker f$.
\end{proof}

\begin{lemma}
\mylabel{lemma-f2e}
  If $f:L\to L$ is a homomorphism such that $f(b)=e$ then $f(L)=f(S)$
  and $f$ is uniquely determined by its values on $S$.
\end{lemma}

\begin{proof}
  We see that if $1$ is the ring identity of $E$ then
  $f(1)=f(cb)=f(c)$, hence $f(1)$ is torsion.
  Since $E$ is torsion free, Lemma \ref{lemma-E} implies that
  $f(1)=e$ and therefore
  $f(E)=\{e\}$ since any endomorphism of $E$
  is either trivial or injective (Theorem \ref{thm-ering}(5)).
  Our claim is proved since $L$ is generated by $S$ and $E$.
\end{proof}

\begin{lemma}
\mylabel{lemma-conj}
  For any homomorphism
  $f:S\to L$ there exists an inner automorphism $c_g$ of $L$ such
  that $c_gf(S)\subseteq S$, that is, for some $f_S:S\to S$,
  the following diagram commutes:
  $$
  \xymatrix{
    S \ar[r]^f \ar[d]^{f_S} & L \ar[d]^{c_g} \\
    S \ar[r]^\eta & L
  }
  $$
  where $\eta:S\to L$ is the inclusion.
\end{lemma}

\begin{proof}
  Since $S$ is finite Lemma \ref{lemma-finite-conjugate} tells us that
  its action on $T_L$
  has a fixed point. Hence, as there are no nontrivial homomorphisms $S\to E$,
  we see that $f(S)$ is conjugate to
  $\presup{k}f(S)\subseteq K$.
  The group $K$ again acts on a tree $T_K$,
  hence by Lemma \ref{lemma-finite-conjugate} again
  we see that $\presup{k}f(S)$ is conjugate to a subgroup
  of $S$ or $M$.

  It is enough to show that if $f(S)$ is isomorphic to
  $G\subseteq M$ then $G$ is conjugate in $M$ to
  a subgroup of $N$. Let $H$ be the centralizer of $A$ in $N$.
  Then
  $$
  M/H\cong(C/A)\rtimes(N/H)
  $$
  Now $N/H$ is isomorphic to a subgroup of $\mathbb{Z}/(p-1)$.
  Notice that if $x\in M$ is an element of order $p$, then $x\in
  H$ (actually, by Property {\propna}, $x\in A$, but we do not
  need this). Since by Property {\prophomzp}, $G$ contains no
  elements of order $p^2$, the image of $G$ in $M/H$ is a
  $p'$-group. Hence, by Hall's theorem, this image is conjugate
  in $M/H$ to a subgroup of $N/H$, and it follows that $G$ is
  conjugate in $M$ to a subgroup of $N$.
\end{proof}

\section{Proof of the main theorem}
\label{section-proof}

\begin{proposition}
\mylabel{proposition-exist}
  For any homomorphism $f:S\to L$ there exists a homomorphism
  $f':L\to L$
  which closes the diagram
  $$
  \xymatrix{
    S \ar[dr]^f \ar[rr]_\eta && L \ar@{-->}[dl]^{f'} \\
    & L
  }
  $$
  where $\eta$ is the inclusion.
\end{proposition}

\begin{proof}
  It is enough to prove our claim for $f$ composed with some
  automorphism of $L$; hence
  by Lemma \ref{lemma-conj} we may assume that
  $f=\eta{f_S}$, so it is enough to close the following diagram:
  $$
  \xymatrix{
    S \ar[d]^{{f_S}} \ar[r]^{\eta} & L  \ar@{-->}[d]^{{f'}}\\
    S \ar[r]^\eta & L
  }
  $$
  If ${f_S}$ is an automorphism then it is an inner
  automorphism, hence we define $f'$ as the inner
  automorphism of $L$ determined by the same element.
  Otherwise by Property {\propautker} we have
  ${f_S}(a)=e=f_S(b)$, hence $f_S(A)=\{e\}$.
  Since, by definition, $N/A=M/C$ we see that
  ${f_S}$ restricted to $N$
  extends to $f':M\to L$ so that $f'(c)=e$ and we obtain an
  extension $f':K=M*_NS\to L$. Since $f'(cb)=e$
  we extend it further to $E$ by defining $f'(E)=\{e\}$.
\end{proof}

\begin{lemma}
\mylabel{lemma-unique}
  If a homomorphism $f:L\to L$ is the identity on $S$ then $f$ is
  the identity on $L$.
\end{lemma}

\begin{proof}
  Since $f$ is the identity on $A\subseteq S$
  the kernel of $f$ must trivially intersect $C$ so that
  the order of $f(c)$ is $p^2$. Since no element of $E$ or
  $S$ has order $p^2$, Lemma \ref{lemma-finite-conjugate}
  implies that $f(c)$ is conjugate in $L$ to a
  generator of $C$. We have
  $$
  f(c)=gc^mg^{-1}
  $$
  for some $g\in L$ and $m$ not divisible by $p$.
  Since $a=c^p$ we see that
  $a = f(a) = ga^mg^{-1}$, that is, $g$ normalizes
  $\langle a\rangle$,
  hence by Lemma \ref{lemma-A-normalizer} $g\in K=M*_NS$.

  Since $A$ is the unique subgroup of $N$ of order $p$, Lemma
  \ref{lemma-normal-subgroup} implies that $g\in
  N_M(A)*_NN_S(A)=M*_NN=M$, hence $g$ normalizes $C$. Possibly
  changing the value of $m$, we obtain
  $$
    f(c)=c^m
  $$

  Since $cb\in E$, Lemma \ref{lemma-E} implies that
  $f(cb)\in \presup{k}E$ for some $k\in L$.
  Let $\pi:E*_Z K\to E/Z$ be the amalgamation of the projection
  $E\to E/Z$ and a homomorphism which sends $K$ to $e$.
  Since $c$ and $b$ are in $K$ we see that
  $$
    f(cb) \in \presup{k}E\cap\ker\pi
  $$
  but the kernel of $\pi$ restricted to $\presup{k}E$ is
  $\presup{k}\langle cb \rangle$, hence
  $$
    c^mb = f(cb) = k(cb)^nk^{-1}
  $$
  for some $n \in \mathbb{Z}$ and $k\in L$.
  Since $c^mb$ and $(cb)^n$ belong to $K$,
  Lemma \ref{lemma-conjugate-elements} implies that we may choose
  $k$ to be in $K$.

  Lemma \ref{lemma-cyclically-reduced-conjugate}, applied to
  $c^mb=k(cb)^nk^{-1}$, implies that
  $n\in\{-1,1\}$, and since $b^{-1}=b$, we have
  $$
    c^mb = kc^nbk^{-1}
  $$
  where $k\in N$. Therefore
  $$
    c^m = (kc^nk^{-1})(kbk^{-1}b^{-1})
  $$
  Since $c^m$ and $kc^nk^{-1}$ belong to $C$ we have
  $kbk^{-1}b^{-1}\in C\cap S=A$, and therefore Remark {\propxxa}
  implies $k=e$, so $m=n$, that is, $f(c)=c$ or
  $f(c)=c^{-1}$. Since the latter implies $a=f(a)=a^{-1}$, which is
  impossible by Property {\propaorder}, we have
  $$
    f(c)=c
  $$
  Since $K$ is generated by $S$ and $c$ we see that $f$ is an
  identity on $K$.

  By Lemma \ref{lemma-E} we know that $f$ maps $E$ to
  $\presup{k}E$ for some $k \in L$. Since $cb\in K$
  we know that $f(cb)=cb$, hence $cb\in E\cap\presup{k}E$.
  Thus $E=\presup{k}E$ by Lemma \ref{lemma-trivial-intersection}.
  The kernel of $f-\mbox{id}$ restricted to $E$
  contains $cb$, hence by Theorem \ref{thm-ering}(5) we
  see that $f$ is the identity on $E$.
\end{proof}

\begin{proposition}
\mylabel{proposition-unique}
  If $f,g:L\to L$ are two homomorphisms that coincide on $S$
  then they are equal.
\end{proposition}
\begin{proof}
  If the kernel of $f$ intersects $S$ nontrivially then by Lemma
  \ref{lemma-conj} and Property {\propautker} we have $g(b)=f(b)=e$, hence
  $f=g$ by Lemma \ref{lemma-f2e}.

  If the kernel of $f$ intersects $S$ trivially then by Lemma
  \ref{lemma-conj} there exists an automorphism $h:L\to L$
  such that $hf$ induces an automorphism of $S$.
  Since $S$ has no outer automorphisms, we may assume
  that $hf$ is the identity on $S$.
  Since $hf$ and $hg$ are identities on $S$ our claim
  follows from Lemma \ref{lemma-unique}.
\end{proof}

The following theorem summarizes the results of the paper.

\begin{theorem}
\mylabel{theorem-main}
  If a finite group $S$ has no outer automorphisms and
  satisfies Properties P1--P8 of Section $4$ then
  for any cardinal number $\kappa$ there exists a closed
  inclusion $S\to L$ such that:
  \begin{itemize}
  \item[(1)]
    The cardinality of $L$ is not less than $\kappa$.
  \item[(2)]
    There exists an abelian subgroup $E\subseteq L$ such that
    $L$ is generated as a group by $E$
    and the image of $S$ in $L$.
  \item[(3)]
    There exists an epimorphism $\pi:L\to E/Z$
    where $Z$ is an infinite cyclic subgroup of $E$.
    The composition $S\to L\to E/Z$ is trivial.
  \end{itemize}
\end{theorem}
\begin{proof}
  Propositions
  \ref{proposition-exist} and \ref{proposition-unique}
  imply that the inclusion $\eta:S\to L$ described in Section
  \ref{section-construction} is closed.
  Items (1)--(3) follow immediately from the construction,
  presented in Section \ref{section-construction}.
\end{proof}

\section{Example: closed embeddings of the Mathieu group.}
\label{section-examples}

In this section we prove, as a consequence of Theorem
\ref{theorem-main}, the existence of closed embeddings of the
Mathieu group $M_{11}$ into arbitrarily large groups.

\begin{example}
\label{example-mathieu}
  Let $S=M_{11}$ be the Mathieu group.
  For any cardinal number $\kappa$ there exists a closed embedding
  $M_{11}\to L$
  such that the cardinality of $L$ is at least $\kappa$.
  The group $L$ is generated by $M_{11}$ and an abelian subgroup
  $E\subseteq L$.
  The group $L$ has an abelian quotient isomorphic to $E/Z$
  where $Z$ is an infinite cyclic subgroup of $E$,
  in particular $L$ is far from being simple.
\end{example}
\begin{proof}
  We use \cite[p. 262]{gorenstein}. Let $a$ be an
  element of order $11$ in $S=M_{11}$. Then $N \mathrel{\mathop:}=
  N_S(\langle a\rangle)$ is a group of order $11 \cdot 5$ and
  $C_S(a) = \langle a\rangle$. Let $b \in S$ be an involution. Then
  Properties P1--P7 are immediate. Further if we choose $b$
  so that $b$ does not normalize any $5$-Sylow subgroup of $N$ (the
  existence of such a $b$ follows from an easy counting argument),
  then property {\propbnb} holds (since $b$ normalizes $N\cap
  \presup{b}N$).
\end{proof}


\begin{thebibliography}{00}

\bibitem{aschbacher} M. Aschbacher {\em On a question of Farjoun},
 Finite groups 2003, 1--27, Walter de Gruyter GmbH \& Co. KG, 2004.

\bibitem{badzioch} B. Badzioch and M. Feshbach
  {\em A note on localizations of perfect groups},
  Proc. Amer. Math. Soc. {\bf 133} (2005), 693--697.

\bibitem{brown} K.S. Brown {\em Cohomology of groups},
  Springer-Verlag 1982.

\bibitem{casacuberta-survey} C. Casacuberta
  {\em On structures preserved by idempotent transformations
  of groups and homotopy types}, Crystallographic Groups and Their
  Generalizations (Kortrijk, 1999), Contemp. Math. {\bf 262},
  AMS, Providence, 2000, 39--69.

\bibitem{erings} M. Dugas, A. Mader and C Vinsonhaler
  {\em Large $E$-rings exist},
  J. Algebra {\bf 108} (1987), 88--101.

\comment{
\bibitem{gap} The GAP Group,
  {\em GAP -- Groups, Algorithms, and Programming},
  Version 4.4.4; 2004 (http://www.gap-system.org)
}

\bibitem{gs} R. G\"{o}bel, J. Rodr\'{i}guez and S. Shelah
  {\em Large localizations of finite simple groups},
  J. Reine Angew. Math. {\bf 550} (2002), 1--24.

\bibitem{grs} R. G\"{o}bel and S. Shelah
  {\em Constructing simple groups for localizations},
  Comm. Algebra {\bf 30} (2002) 809--837.

\bibitem{gorenstein}
  D. Gorenstein, R. Lyons and R. Solomon
  {\em The classification of the finite simple groups}. Number 3. Part I.
  Chapter A. {\em Almost simple $K$-groups},
  Mathematical Surveys and Monographs, 40.3. American Mathematical
  Society, Providence, RI, 1998.

\bibitem{libman} A. Libman
  {\em A note on the localization of finite groups},
  J. Pure Appl. Algebra {\bf 148} (2000), 271--274.

\bibitem{libman-nilpotent} A. Libman
  {\em Cardinality and nilpotency of localizations of groups and
  $G$-modules}, Israel J. Math. {\bf 117} (2000), 221--237.

\bibitem{magnus} W. Magnus, A. Karrass and D. Solitar
 {\em Combinatorial group theory},
 Interscience Publishers [John Wiley \& Sons, Inc.], 1966.

\bibitem{rsv-simple} J. Rodr\'iguez, J. Scherer and A. Viruel
  {\em Non-simple localizations of finite simple groups},
  J. Algebra {\bf 305} (2006), 765--774.

\bibitem{rsv-perfect} J. Rodr\'iguez, J. Scherer and A. Viruel
  {\em Preservation of perfectness and acyclicity: Berrick and
  Casacuberta's universal acyclic space localized at a set of primes},
  Forum Math. {\bf 17} (2005), 67--75.

\bibitem{schultz} P. Schultz,
  {\em The endomorphism ring of the additive group of a ring},
  J. Austral. Math. Soc. {\bf 15} (1973), 60--69.


\bibitem{serre} J.-P. Serre
  {\em Trees}, Springer-Verlag, 1980.


\end{thebibliography}
\end{document}